\documentclass[12pt]{amsart}

\usepackage{graphicx}
\usepackage{amsthm}
\usepackage{amsmath, amssymb}
\swapnumbers

\theoremstyle{plain}
\newtheorem{Thm}{Theorem}

\newtheorem{Coro}[Thm]{Corollary}
\newtheorem{Lem}[Thm]{Lemma}

\theoremstyle{definition}
\newtheorem{Def}[Thm]{Definition}

\begin{document}

\title{Topological graph clustering with thin position}

\author{Jesse Johnson}
\address{\hskip-\parindent
        Department of Mathematics \\
        Oklahoma State University \\
        Stillwater, OK 74078 \\
        USA}
\email{jjohnson@math.okstate.edu}


\thanks{This project was supported by NSF Grant DMS-1006369 and was inspired by the workshop The Geometry of Large Networks at the American Institute of Mathematics in November, 2011.}

\begin{abstract}
A clustering algorithm partitions a set of data points into smaller sets (clusters) such that each subset is more tightly packed than the whole. Many approaches to clustering translate the vector data into a graph with edges reflecting a distance or similarity metric on the points, then look for highly connected subgraphs. We introduce such an algorithm based on ideas borrowed from the topological notion of thin position for knots and 3-dimensional manifolds.
\end{abstract}

\maketitle

\section{Introduction}

Data mining is the search for patterns and structure in large sets of (often high dimensional) data. This data is generally in the form of vectors, which one thinks of as points sampled from some underlying probability measure. This probability measure may be a sum of probability measures corresponding to different types of points, in which case one expects the different types of points to form geometrically distinguishable \textit{clusters}.

Clustering algorithms fall into a number of categories determined by the assumptions they make about the underlying probability measures and their approach to searching for clusters. (A good introduction is Everitt's book~\cite{clusters}.) The $K$-means algorithm~\cite{kmeans} assumes that the underlying measures are Gaussian distributions centered at $K$ points throughout a Euclidean space. This has been generalized in a number of ways, but all assume fairly restrictive models and are less effective for high dimensional data where the Euclidean metric is less meaningful.

Hierarchical clustering algorithms~\cite{clink, slink} arrange the data points by building a tree, placing each point in the tree based on its relation to the previously added nodes. This allows much more flexibility of the model and the metric, but for many of these algorithms, the final structure is dependent on the order in which the tree is constructed. 

Graph partitioning algorithms translate the data points into a graph with weighted edges in which the weights reflect the similarity between points in whatever metric is most natural for the given type of data. The clusters are defined by subgraphs that can be separated from the whole graph by removing relatively few edges. Graph clusters should come very close to realizing the Cheeger constant for the graph and there are algorithms for finding them based on linear programming~\cite{linear} as well as spectral analysis of the Laplacian of the adjacency matrix~\cite{spectral}. Carlsson and Memoli~\cite{carlmem} recently introduced a hierarchical cluster method based on encoding the data as a simplicial complex (a generalization of a graph), giving the algorithm a very strong grounding in topology.

In the present paper we define another topological approach to graph clustering inspired by the idea of \textit{thin position} for knots and 3-manifolds~\cite{gabai, st:thin}. In the context of 3-manifolds, thin position determines minimal genus Heegaard splittings, which are related to minimal surfaces~\cite{pittsrub} and the Cheeger constant~\cite{lackenby}. As we will show, thin position translates quite naturally to graph partitioning/clustering.

The resulting algorithm is gradient-like in the sense that it begins with an ordering of the vertices of the graph, defines a ``width'' of the ordering, then looks for ways to find ``thinner'' orderings. Once there are no more possible improvements, there is a simple criteria that decides if the first $k$ vertices should be considered a cluster for any $k < N$. As with any gradient method, there is the potential to get caught in a local minimum if one starts with a bad initial ordering. However, we show in Lemma~\ref{effective1lem} that for any pinch cluster (defined below), there is some initial ordering that will guarantee the algorithm finds it. Thus one would expect to improve the performance of the algorithm by running it repeatedly with different random initial orderings. Its performance on actual data sets will be examined in future papers.

We define what we will mean by a cluster in Section~\ref{defsect}. The reordering portion of the algorithm is described in Section~\ref{orderingsect} and the interpretation of the final ordering to find clusters is described in Section~\ref{clustersect}. Section~\ref{effectivesect} contains the proof of Lemma~\ref{effective1lem}.

\section{Pinch Clusters}
\label{defsect}

Let $G = (V, E)$ be a graph, where $V$ is the set of vertices and $E$ the set of (weighted or unweighted) edges. We will assume throughout the paper that each edge has distinct endpoints. For $A \subset V$ a vertex subset, the \textit{boundary} $\partial A \subset E$ is the set of edges with one endpoint in $A$ and the other endpoint in the complement $V \setminus A$. The \textit{size} of the boundary, $|\partial A|$, is the sum of the weights of these edges. (For a graph with unweighted edges, $|\partial A|$ is the number of such edges.) There is no universally agreed upon definition of a cluster, but roughly speaking one would want it to have a relatively small boundary relative to the number/weights of the edges that do not cross its boundary. For the purposes of this paper, we use the following definition:

\begin{Def}
A \textit{pinch cluster} is a set of vertices $A \subset V$ with the property that for any sequence of vertices $w_1,\dots,w_m$, if adding $w_1,\dots,w_m$ to $A$ or removing $w_1,\dots,w_m$ from $A$ creates a set with smaller boundary then for some $k < m$, adding/removing $w_1,\dots,w_k$ to/from $A$ creates a set with strictly larger boundary. 
\end{Def}

In other words, if you add a sequence of vertices to $A$ or remove  a sequence of vertices from $A$ one at a time, the boundary size must increase before it decreases. This definition is complicated by the fact that adding or removing a vertex may keep the boundary the same size. Adding/removing a single vertex to/from a pinch cluster $A$ cannot strictly decrease the size of its boundary, so a pinch cluster will satisfy the following two conditions:

\begin{enumerate}
\item For every vertex $v \in A$, the sum of the edge weights from $v$ to other vertices in $A$ is greater than or equal to the sum of the edge weights from $v$ to vertices outside $A$.
\item For every vertex $v \notin A$, the sum of the edge weights from $v$ to other vertices outside of $A$ is greater than or equal to the sum of the edge weights from $v$ to vertices inside $A$.
\end{enumerate}

The vertices inside a pinch cluster are more connected to each other than to vertices outside the pinch cluster, while vertices outside are more connected to each other than to the vertices inside. This is illustrated in Figure~\ref{longneckfig}, where we can cut the graph roughly in half in a number of places, each time by cutting three edges. The vertices to the left of the middle cut, for example, form a pinch cluster because while we can add or remove two vertices without increasing its boundary, if we add/remove any further vertices, the boundary increases.
\begin{figure}[htb]
  \begin{center}
  \includegraphics[width=3in]{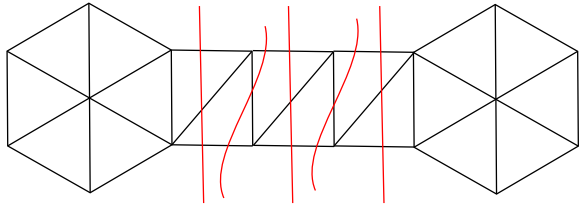}
  \caption{Two distinct clusters with ambiguous points in between.}
  \label{longneckfig}
  \end{center}
\end{figure}

\section{Orderings}
\label{orderingsect}

Let $N = |V|$ be the number of vertices in the graph $G$. An \textit{ordering} of $V$ is a bijection $o : V \rightarrow [1,N]$ where $[1,N]$ represents the integers from $1$ to $N$ (rather than the interval of real numbers). In cases where an ordering is understood, we will write $v_i = o^{-1}(i)$ so that $V = \{v_1,\dots,v_N\}$. For each $i$, let $A_i = \{v_1,\dots,v_i\} \subset V$. The \textit{width at level $i$} is the size of the boudary $|\partial A_i|$, or equivalently the sum of the weights of the edges between all vertices $v_j$, $v_k$ with $j \leq i < k$.

Our goal will be to find an ordering such that some $A_i$ is a pinch cluster.  For such a set $A_i$, the width at level $i$ will be smaller than the widths at nearby values of $i$, or in other words $i$ will be a local minimum with respect to width. The main insight offered by thin position is that to find local minima that are pinch clusters, we must minimize the widths of all the sets $A_i$, particularly the local maxima.

Let $b_i = |\partial A_i|$ for each $i$. The \textit{width} of the ordering $o$ is the vector $w(o) = (w_0,\dots,w_N)$ where the values $w_i$ consist of the values $b_i$, rearranged into non-increasing order. We will compare the widths of different orderings using lexicographic (dictionary) ordering: Given width vectors $w = (w_i)$ and $u = (u_i)$ (with non-increasing components) say $w < u$ if there is a value $i$ such that $w_i < u_i$ while $w_j = u_j$ for every $j < i$. In other words, we compare the entries of the vectors, starting from the first, until we find a component where they disagree. Then we will decide which is smaller based on this component.

Figure~\ref{thinex1fig} shows two orderings on the same graph and the induced widths. The ordering on the left has $w = (4,3,3,2,2)$, while the ordering on the right has $w = (3,2,2,2,2)$. Thus the ordering on the right is thinner. Note that we can compare vectors of different lengths by appending zeros to the end of the shorter vector.
\begin{figure}[htb]
  \begin{center}
  \includegraphics[width=5in]{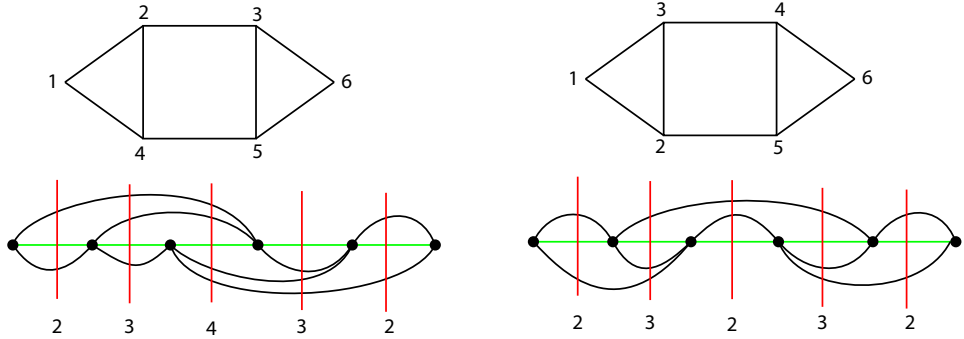}
  \caption{Two different ordering of the vertices.}
  \label{thinex1fig}
  \end{center}
\end{figure}

We will say that an ordering $o$ is \textit{(globally) thin} if for every other ordering $o'$ on $G$, we have $w(o) \leq w(o')$. It turns out that in a globally thin ordering, each value $i$ at which the width is locally minimal defines a pinch cluster. Finding a globally thin ordering would be computationally very expensive (probably NP hard) but luckily we will see that the local minima will still define pinch clusters if the ordering has a closely related property which we will define below.

A \textit{flat} in an ordering $o$ is an interval $F = [i,j]$ such $b_i = b_{i+1} = \cdots = b_j$ but the widths at $i-1$ and $j+1$ are different from the widths $b_i = b_j$. Note that a flat may consist of a single element $F = [i,i]$. A flat is \textit{(locally) maxmimal} if the widths of $o$ at $i-1$ and $j+1$ are both lower than the width at any level in $F$. Similarly, a flat is \textit{(locally) minimal} if the widths of $o$ at $i-1$ and $j+1$ are both greater than the width at any level in $F$. We will say that $i \in [1,N]$ is a \textit{(local) minimum/maximum} if it is contained in a locally minimal/maximal (respectively) flat.

The \textit{slope} of a vertex $v \in V$ with respect to subset $A \subset V$, written $s_A(v)$, is the sum of the edge weights from $v$ to vertices in $V \setminus A$ minus the sum of the edge weights to vertices in $A$. If $v \notin A$ then this is the amount that $|\partial A|$ will increase if we add $v$ into the set. If $v \in A$, this is the amount $|\partial A|$ will decrease if we remove it from $A$.  We get from $A_{i-1}$ to $A_i$ by adding $v_i$ so $|\partial A_i| = |\partial A_{i-1}| + s_{A_i}(v_i)$ and we will abbreviate $s_{i,j} = s_{A_i}(v_j)$. (Note that $s_{i,i} = s_{i-1,i}$.) Given $b_i = |\partial A_i|$ as above, we have:
$$b_i = b_{i-1} + s_{i,i}$$

Let $(a_{k,\ell})$ be the adjacency matrix for $G$, i.e.\ $a_{k,\ell}$ is the weight of the edge from $v_k$ to $v_\ell$, (1 for unweighted edges) or zero if there is no edge. Note that we can calculate all the slopes $\{s_{i,j}\}$ in $O(N^2)$ time by calculating the row sums of the adjacency matrix (the index of each vertex), then going through each row and subtracting $2a_{k,\ell}$ from the row sum for each $\ell$. In particular, $s_{i,j}$ is non-increasing for fixed $j$.

\begin{Def}
Let $F = [i,j]$ be a locally maximal flat in an ordering $o$. Let $[i^-,j^-]$ and $[i^+, j^+]$ be the unique locally minimal flats with $j^- < i$, $j < i^+$ such that there are no minima strictly between $j^-$ and $i^+$. Note that this implies $[i,j]$ is the only locally maximal flat in this interval. We will say that $i$ is \textit{weakly reducible} if for some value $k$ with $j^- < k \leq i^+$, one of the following conditions hold:
 \begin{enumerate}
\item For $k \leq i$, $s_{i,k} > 0$ and $s_{i,k} - s_{i,i+1} - 2a_{k,i+1} > 0$ or 
\item For $k > j$, $s_{j,k} < 0$ and $-s_{j,k} + s_{j,j} - 2a_{k,j} > 0$.
\end{enumerate}
A maximum that is not weakly reducible is called \textit{strongly irreducible}. \end{Def}

We will show that given an ordering with a weakly reducible maximum, we can find an ordering with lower complexity as follows: A permutation $\sigma$ of $[0,N]$ is a \textit{shift} if it is of the form $(j\ i\ (i+1)\ \cdots\ j-1)$ or $((i+1)\ (i+2)\ \cdots\ j\ i)$ for $i < j$. In other words, a shift is cyclic on a set of consecutive integers and moves each integer to the left or right by one. If we compose the ordering $o$ with one of these shifts, we will say that we \textit{shift $j$ to $i$} or \textit{shift $i$ to $j$}, respectively.

\begin{Lem}
\label{strongmaxlem}
If an ordering $o$ has a weakly reducible maximum with $i$,$j$,$k$ as in the definition, then shifting $k$ to $i+1$ (in the case $k \leq i$) or shifting $k$ to $j$ (in the case $k > j$) produces an ordering $o'$ with $w(o') < w(o)$.
\end{Lem}

\begin{proof}
Without loss of generality assume condition (1) from the definition holds, i.e.\ $k \leq i$. We will shift $k$ to $i+1$. For each $\ell$ with $k \leq \ell \leq i$, the shift has the effect of removing $v_k$ from $A_\ell$ and adding the vertex $v_{\ell+1}$. Removing $v_k$ decreases the boundary of $A_\ell$ by $s_{\ell,k}$. It also increases the slope of $v_{\ell+1}$ with respect to level $\ell$ by $2a_{k,\ell+1}$. Thus adding $v_{\ell+1}$ to $A_i$ increases the boundary of the set by $s_{\ell, \ell+1} + 2a_{k,\ell+1}$. If $A'_\ell$ is the set after the shift and $b'_\ell$ is its width, we have
$$b'_\ell = b_\ell - s_{\ell,k} + s_{\ell, \ell+1} + 2a_{k,\ell+1}$$
For $\ell = i$, this is $b'_i = b_i - s_{i,k} + s_{i, i+1} + 2a_{k,i+1}$. By assumption, $s_{i,k} - s_{i, i+1} - 2a_{k,i+1} > 0$ so we conclude that $b'_i < b_i$. 

For the remaining values of $\ell$, we do not have such a condition, so the boundary of $A_k$ could potentially increase. However, we have $b_i = b_{\ell} + s_{\ell,\ell+1} + \cdots + s_{i-1,i}$. Because there are no maxima between $k$ and $i$, all these slopes are nonnegative, so
$$b_{\ell} + s_{\ell,\ell+1} \leq b_i$$
Substituting this into the above formula, we find
$$b'_\ell \leq b_i - s_{\ell,k} + 2a_{k,\ell+1}$$
The slopes satisfy the formula $s_{i,k} = s_{l,k} - 2a_{k,\ell+1} - \cdots - 2a_{k,i}$. Since the adjacency matrix is non-negative, this implies
$$s_{i,k} \leq s_{l,k} - 2a_{k,\ell+1}$$
Negating both sides and substituting into the previous formula, we find
$$b'_\ell \leq b_i - s_{i,k}$$
By assumption, $s_{i,k} > 0$ so $b'_\ell < b_i$. Thus while the maximum value of $b_\ell$ for $k \leq \ell \leq i$ is $b_i$, the maximum value for $b'_i$ in this range is strictly smaller. Because of the lexicographic ordering, this implies $w(o') < w(o)$.
\end{proof}

\section{Finding clusters}
\label{clustersect}

If all the maxima in an ordering $o$ are strongly irreducible then there may be further shifts that reduce the width by rearranging the vertices within each interval between consecutive minima and maxima. However, such shifts do not seem to be helpful for finding pinch clusters, so we will not seek to minimize the width beyond applying Lemma~\ref{strongmaxlem}. Because there are finitely many possible widths and each shift strictly reduces the width, we are guaranteed to find an ordering that cannot be further reduced, i.e.\ one in which all the maxima are strongly irreducible. Such an ordering will be called \textit{strongly irreducible}. The significance of this definition comes from the following:

\begin{Thm}
\label{thinlevelthm}
If $k$ is a local minimum of a strongly irreducible ordering $o$ for $G$ then $A_k$ is a pinch cluster.
\end{Thm}

Note that the complement of a pinch cluster is also a pinch cluster, i.e.\ if $A_k = \{v_1,\ldots,v_k\}$ is a pinch cluster then so is $\{v_{k+1},\ldots, v_N\}$. To simplify the discussion, however, we will only talk about the single pinch cluster $A_k$.

Before proving Theorem~\ref{thinlevelthm}, we introduce the following terminology, which should make the proofs slightly easier to follow.

\begin{Def}
A set $A \subset V$ is \textit{pinch convex} if it has the property that for any sequence of vertices $w_1,\dots,w_m \in V \setminus A$, if adding $\{w_1,\dots,w_m\}$ to $A$ creates a set with smaller boundary then for some $k < m$, adding $\{w_1,\dots,w_k\}$ to $A$ must create a set with strictly larger boundary.

A set $A \subset V$ is \textit{pinch concave} if it has the property that for any sequence of vertices $w_1,\dots,w_m \in A$, if removing $\{w_1,\dots,w_m\}$ from $A$ creates a set with smaller boundary then for some $k < m$, removing $\{w_1,\dots,w_k\}$ from $A$ creates a set with strictly larger boundary.
\end{Def}

By definition, a vertex subset is a pinch cluster if and only if it is both pinch convex and pinch concave. Note that the complement of a pinch convex set is pinch concave and vice versa.  To remember which is which, note that a convex polygon has the property that you cannot add to it without increasing its boundary. The complement of a convex polygon (which we can consider concave) has the property that you cannot remove a portion of it without increasing its boundary.

\begin{proof}[Proof of Theorem~\ref{thinlevelthm}]
Let $k$ be in a local minimum for a strongly irreducible ordering $o$ and assume for contradiction that $A_k$ is not a pinch cluster, i.e.\ $A_k$ fails to be either pinch concave or pinch convex. If $A_k$ fails to be pinch convex then its complement $V \setminus A_k$ fails to be pinch concave. In this case, by reversing the order of $o$ and replacing $A_k$ with $V \setminus A_k$ we can, without loss of generality, consider the case when $A_k$ fails to be pinch concave.

Let $v_{j_1},\dots,v_{j_m}$ be a sequence of vertices in $V$ such that removing the first $m-1$ vertices from $A_k$ does not change the size of the boundary, but removing the final vertex $v_{j_m}$ does reduce the boundary of the set. 

By assumption, $v_{j_\ell} \in A_k$ so for each $\ell \leq m$, $j_\ell \leq k$. Because removing $v_{j_1}$ from $A_k$ does not increase $|\partial A_k|$, we have $s_{k,{j_1}} \geq 0$. Since $k$ is a local minimum, $s_{k,k} \leq 0$. Thus if $j_1 = k$ then $s_{k,k} = 0$, $m > 1$ and $k-1$ is also a minimum so we can restart the argument with $A_{k-1}$ and the sequence of vertices $v_{k_2},\dots, v_{k_m}$.

Otherwise, we have a strict inequality $j_1 < k$. The function $s_{i,j}$ is non-increasing for fixed $j$, so for each $i < k$, $s_{i, j_1} \geq s_{k,j_1}$. In particular, $s_{j_1,j_1} \geq 0$, with equality $s_{j_1,j_1} = 0$ precisely when $s_{k,j_1} = 0$ and there are no edges between $v_{j_1}$ and any of the vertices between $j_1$ and $k$.

In the case when $s_{j_1,j_1} = 0$, because there are no edges between $v_{j_1}$ and any vertex between $j_1$ and $k$, shifting $j_1$ to $k$ will not increase the width of the ordering or introduce any new minimal flat (though it may add a vertex to the existing minimal flat). Thus if we perform the shift, we will have a new ordering with the same $A_k$ and $k$ will still be a minimum. Moreover, this puts us in the above case when $j_1 = k$, so we can proceed as above.

By reducing $k$ in this way, we must eventually come to the case when $s_{j_1,j_1} > 0$. Since the width increases at step $j_1$, but decreases at step $k > j_1$, there must be a local maximum $i$ such that $j_1 \leq i < k$. Assume $i$ is the largest such maximum. As noted above, 
$s_{k,j_1} = s_{i,j_1} - 2a_{j_1,i+1} - \cdots - 2a_{j_1,k}$ so $s_{i,j_1} - 2a_{j_1, i+1} \geq s_{k,j_1} \geq 0$. Moreover, since $i$ is the last local maximum before the minimum $k$, we have the strict inequality $s_{i,i+1} < 0$, so $s_{i,j_1} - s_{i,i+1} - 2a_{j_1, i+1} > 0$ and we conclude that $o$ is not strongly irreducible. This contradiction implies that every local minimum of $o$ defines a cluster.
\end{proof}

If an ordering $o$ for $G$ has exactly one locally minimal flat $F$ then Theorem~\ref{thinlevelthm} tells us this minimum defines two clusters for each $i \in F$, one consisting of vertices $\{v_1,\dots,v_i\}$ and the other consisting of $\{v_{i+1}, v_N\}$, where the width at level $i$ is a local minimum. (Though the clusters defined by different values of $i$ in the same flat are not significantly different.)

If $o$ has multiple locally minimal flats then again each minimum cuts the graph into two pinch clusters, but a cluster defined by a minimum in one flat will likely be a union of smaller clusters defined by the other flats. We would like to find these smaller pinch clusters by looking ``between'' the minimal flats.

Let $m_1 < \dots < m_k$ be indices in the locally minimal flats of $o$, such that each minimal flat contains exactly one $m_i$. Define $m_0 = 0$ and $m_{k+1} = N$. Let $B_i = \{v_{m_i+1},\dots,v_{m_{i+1}}\}$ for each $i \leq k$. Then $B_0$ and $B_k$ are both pinch clusters by Theorem~\ref{thinlevelthm}. 

For $B_i$ with $i \neq 0,k$, we can calculate the slope of a vertex $v \in B_i$ with respect to $B_i$ from the adjacency matrix for $G$. If all these slopes are negative then we will show below that $B_i$ is a pinch cluster. If one or more of these slopes is positive then $B_i$ will not be pinch concave because removing such a vertex from the set decreases its boundary. If a slope is non-negative, then $B_i$ may or may not be pinch concave.

Assume there is a vertex in $B_i$ with non-negative slope and let $v_0 \in B_i$ be a vertex with maximal slope with respect to $B_i$. Define $B^1_i = B_i \setminus v_0$. We can again calculate the slopes of the vertices with respect to $B^1_i$. If any of these slopes are non-negative, remove another vertex with maximal slope to find a set $B^2_i$. Repeat this process until it terminates, with a set $B^\ell_i$ such that either $B^\ell_i$ is empty or every vertex in $B^\ell_i$ has negative slope with respect to $B^\ell_i$. Define $B'_i = B^\ell_i$.

\begin{Lem}
\label{middleclusterlem}
If the algorithm terminates with a non-empty set $B'_i$ then this set is a pinch cluster.
\end{Lem}

Before proving this we need, the following technical Lemma:

\begin{Lem}
\label{convexsubsetlem}
Let $B \subset C$ be vertex sets in which $C$ is pinch convex. If there is a sequence of vertices $w_1,\dots,w_m$ such that adding $w_1,\dots,w_m$ to $B$ increases its boundary but adding any proper subset $w_1,\dots,w_\ell$ ($\ell < m$) does not increase its boundary then $w_m \in C$.
\end{Lem}

\begin{proof}
Let $B_\ell = B \cup \{w_1,\dots,w_\ell\}$ and $C_\ell = C \cup \{w_1,\dots,w_\ell\}$ for each $\ell \leq m$. Because $B \subset C$, we have $B_\ell \subset C_\ell$ for each $\ell$, and therefore $s_{C_\ell}(w_{\ell+1}) \leq s_{B_\ell}(w_{\ell+1})$. Thus if we add the first $\ell < m$ vertices to $C$, the boundary will not increase for any $\ell < m$. The final slope $s_{C_{m-1}}(w_m)$ is negative, so if $w_m \notin C$ then adding $w_m$ to $C$ will decrease its boundary. Since we assumed $C$ is pinch concave, we conclude that $w_m$ must already be contained in $C$.
\end{proof}

\begin{Coro}
\label{convexintersectionlem}
The intersection of two pinch convex sets is pinch convex.
\end{Coro}

\begin{proof}
Let $B$ and $C$ be pinch convex sets and assume for contradiction $B \cap C$ is not pinch convex. Then there is a sequence of vertices $w_1,\dots,w_m \notin (B \cap C)$ such that adding $w_1,\dots,w_\ell$ to $(B \cap C)$ does not increase the boundary for $\ell < m$, but adding $w_1,\dots, w_m$ decreases the boundary. Then by Lemma~\ref{convexsubsetlem}, $w_m \in B$ and $w_m \in C$. But by assumption, $w_m \notin B \cap C$. This contradiction completes the proof.
\end{proof}

\begin{proof}[Proof of Lemma~\ref{middleclusterlem}]
By construction, each vertex of $B'_i$ has negative slope with respect to $B'_i$. If we remove any sequence of vertices from $B'_i$, the boundary will increase with the very first vertex. Therefore $B'_i$ is pinch concave and we only need to check that it is also pinch convex. We construct $B'_i$ as the final set in a sequence $B_i, B^1_i,\dots, B^\ell_i = B'_i$ and we will check that each $B^j_i$ is pinch convex. 

The set $B_i$ is the intersection of $A_{m_{i+1}}$ and $V \setminus A_{m_i+1}$, each of which is a pinch cluster and thus pinch convex. So by Corollary~\ref{convexintersectionlem}, $B_i$ is pinch convex. 

Assume for contradiction $B^1_i$ is not pinch convex, so there is a sequence of vertices $w_1,\dots,w_m$ such that adding $w_1,\dots,w_\ell$ to $B^1_i$ does not increase the boundary for $\ell < m$, but adding $w_1,\dots, w_m$ does. Lemma~\ref{convexsubsetlem} implies that the final vertex $w_m$ is in $B_i$, so it must be the unique vertex in $B_i \setminus B^1_i$. 

The slope of each $w_i$ with respect to $B_i$ is less than or equal to its slope with respect to $B^1_i$, so $w_1,\dots,w_{m-1}$ has the property that adding the first $\ell$ of these vertices to $B_i$ does not increase the boundary of the set. By assumption, we have $|\partial(B^1_i \cup \{w_1,\dots,w_m\}| < |\partial B^1_i| \leq |\partial B_i|$ and the set $B^1_i \cup \{w_1,\dots,w_m\}$ is the same as $B_i \cup \{w_1,\dots,w_{m-1}\}$. Thus adding $\{w_1,\dots,w_{m-1}\}$ to $B_i$ reduces its boundary. This contradicts the fact that $B_i$ is pinch convex, so we conclude that $B^1_i$ must be pinch convex. If we repeat the argument for each $B^j_i$, we find that $B'_i$ is pinch convex and therefore a pinch cluster.
\end{proof}

\section{Effectiveness}
\label{effectivesect}

As noted above, the reordering algorithm (like any gradient method) may terminate with a strongly irreducible ordering that does not determine a pinch cluster. (Of course, some graphs will not contain any non-trivial pinch clusters.) We show in this section, however, that the algorithm has the potential to discover any pinch cluster in $G$ if the correct initial ordering is chosen. Thus running the algorithm beginning with a number of different random initial orderings increases one's chances of finding a useful cluster.

We will say that an ordering $o$ on $G$ \textit{discovers} a pinch cluster $A \subset V$ if $A = \{v_1,\dots,v_i\}$ for some local minimum $i$. We show below that for any cluster $A$, there is a strongly irreducible ordering that discoveres $A$. Because the algorithm stops when it finds a strongly irreducible ordering, the algorithm can return $o$, given a well chosen initial ordering.

\begin{Lem}
\label{effective1lem}
If $A \subset V$ is a pinch cluster for $G$ then there is a strongly irreducible ordering for $G$ that discovers $A$.
\end{Lem}

\begin{proof}
Choose an ordering $o$ such that $A = v_1,\dots,v_i$ for $i=|A|$. This ordering may not be strongly irreducible, but by repeatedly applying Lemma~\ref{strongmaxlem}, we can find a strongly irreducible ordering $o'$. 

In the initial ordering $o$, $i$ is a local minimum. Otherwise, the ordering would define a sequence of vertices such that adding or removing them from $A$ would decrease the boundary of the set without increasing it first. The algorithm only shifts vertices within the intervals between consecutive local minima and never shifts a vertex into or across a local minimum. Thus the first shift leaves the vertices of $A$ in the first $i$ slots.  As above, $i$ will again be a local minimum for the new ordering, so the second shift also leaves the vertices of $A$ in the first $i$ slots. By repeating this argument for each shift, we find that in the final strongly irreducible ordering, we again have $A = \{v_1,\dots,v_i\}$.
\end{proof}

The algorithm addressed by Lemma~\ref{middleclusterlem}, which finds clusters within the middle blocks of a strongly irreducible ordering, may appear somewhat haphazard compared to the reordering algorithm, but in fact, it is guaranteed to find a pinch cluster if one exists. 

Given a pinch cluster $A \subset V$, let $C(A)$ be the set that results from starting with $A$ and repeatedly removing any vertex whose slope with respect to the remaining set is zero. (Because $A$ is a pinch cluster, none of the slopes can be strictly positive.) We will call $C(A)$ a \textit{core} of $A$. Note that $|\partial C(A)| = |\partial A|$ and because $A$ is a pinch cluster, it must have a non-empty core. (Note that we say ``a core'' rather than ``the core''. By this definition, there may not be two or more cores depending on the order in which the vertices are removed. However, the proof of Lemma~\ref{effective2lem} can be modified to show that any pinch cluster has a unique core.)

\begin{Lem}
\label{effective2lem}
If $B_i$ contains a pinch cluster $C$ then the algorithm will terminate with a non-empty pinch cluster $B'_i$ containing every core of $C$.
\end{Lem}

\begin{proof}
Let $C = C(A)$ be a core of a pinch cluster contained in $B_i$ and assume for contradiction that some vertex of $C$ is not in the subset $B'_i$ constructed by the algorithm. Since the vertices are removed sequentially, we will let $v$ be the first vertex of $C$ removed from $B^j_i$ by the algorithm, at step $j+1$. So in particular, $C$ is contained in $B^j_i$ and the slope of $v$ with respect to $B^j_i$ is positive. 

Because $C$ is a core of a pinch cluster and $v \in C$, the slope $s_C(v)$ is strictly negative. Because $C \subset B^j_i$, this implies that the slope of $v$ with respect to $B^j_i$ is strictly negative. However, because the algorithm removed $v$ from $B^j_i$, $s_{B^j_i}(v)$ must have been non-negative. This contradiction implies that every vertex of $C$ remains in $B'_i$.
\end{proof}

\bibliographystyle{amsplain}
\bibliography{thingraphs}

\providecommand{\bysame}{\leavevmode\hbox to3em{\hrulefill}\thinspace}
\providecommand{\MR}{\relax\ifhmode\unskip\space\fi MR }
\providecommand{\MRhref}[2]{%
  \href{http://www.ams.org/mathscinet-getitem?mr=#1}{#2}
}
\providecommand{\href}[2]{#2}
\begin{thebibliography}{10}

\bibitem{carlmem}
Gunnar Carlsson and Facundo M{\'e}moli, \emph{Multiparameter hierarchical
  clustering methods}, Classification as a tool for research, Stud.
  Classification Data Anal. Knowledge Organ., Springer, Berlin, 2010,
  pp.~63--70. \MR{2722123}

\bibitem{clink}
D.~Defays, \emph{An efficient algorithm for a complete link method}, Comput. J.
  \textbf{20} (1977), no.~4, 364--366. \MR{0478804 (57 \#18277)}

\bibitem{spectral}
W.~E. Donath and A.~J. Hoffman, \emph{Lower bounds for the partitioning of
  graphs}, IBM J. Res. Develop. \textbf{17} (1973), 420--425. \MR{0329965 (48
  \#8304)}

\bibitem{clusters}
Brian~S. Everitt, Sabine Landau, Morven Leese, and Daniel Stahl, \emph{Cluster
  analysis}, 5th ed., Wiley, 2011.

\bibitem{kmeans}
J.~A. Hartigan and M.~A. Wong, \emph{A k-means clustering algorithm}, Journal
  of the Royal Statistical Society, Series C (Applied Statistics) \textbf{28}
  (1979), no.~1, 100--108.

\bibitem{lackenby}
Marc Lackenby, \emph{Heegaard splittings, the virtually {H}aken conjecture and
  property {$(\tau)$}}, Invent. Math. \textbf{164} (2006), no.~2, 317--359.
  \MR{2218779 (2007c:57030)}

\bibitem{linear}
M.~Minoux and E.~Pinson, \emph{Lower bounds to the graph partitioning problem
  through generalized linear programming and network flows}, RAIRO Rech.
  Op\'er. \textbf{21} (1987), no.~4, 349--364. \MR{932184 (89e:05159)}

\bibitem{pittsrub}
Jon~T. Pitts and J.~H. Rubinstein, \emph{Existence of minimal surfaces of
  bounded topological type in three-manifolds}, Miniconference on geometry and
  partial differential equations ({C}anberra, 1985), Proc. Centre Math. Anal.
  Austral. Nat. Univ., vol.~10, Austral. Nat. Univ., Canberra, 1986,
  pp.~163--176. \MR{857665 (87j:49074)}

\bibitem{st:thin}
Martin Scharlemann and Abigail Thompson, \emph{Thin position for
  {$3$}-manifolds}, Geometric topology ({H}aifa, 1992), Contemp. Math., vol.
  164, Amer. Math. Soc., Providence, RI, 1994, pp.~231--238. \MR{1282766
  (95e:57032)}

\bibitem{slink}
R.~Sibson, \emph{S{LINK}: an optimally efficient algorithm for the single-link
  cluster method}, Comput. J. \textbf{16} (1973), 30--34. \MR{0321382 (47
  \#9915)}

\end{thebibliography}

\end{document}